\documentclass[12pt]{article}%[10pt]{article}%\documentclass
\usepackage{amssymb,latexsym,amsmath,amsthm,amsfonts, enumerate}
\usepackage{color}
\usepackage[all]{xy}

\usepackage{hyperref}

\def\Hom{\mathop{\rm Hom}\nolimits}
\def\Ext{\mathop{\rm Ext}\nolimits}

\def\Fac{\mathop{\rm Fac}\nolimits}

\def\mod{\mathop{\rm mod}\nolimits}

\def\End{\mathop{\rm End}\nolimits}
\def\tilt{\mathop{\rm tilt}\nolimits}

\def\ind{\mathop{\rm ind}\nolimits}
\def\tilt{\mathop{\rm \tau\makebox{-}tilt}\nolimits}
\def\stilt{\mathop{\rm s\tau\makebox{-}tilt}\nolimits}

\def\rad{\mathop{\rm rad}\nolimits}
\def\sbrick{\mathop{\rm sbrick}\nolimits}
\def\sf{\mathop{\rm sf}\nolimits}
\def\f{\mathop{\rm f}\nolimits}

\usepackage{ulem}
\begin{document}

%% ----------------------------------------------------------------------
\newcommand{\nc}{\newcommand}

%\def\cosilting{quasi-cotilting }
%\def\Cosilting{Quasi-cotilting }

%%%%%%%%FROM latexexam.tex

\newtheorem{theorem}{Theorem}[section]
\newtheorem{proposition}[theorem]{Proposition}
\newtheorem{lemma}[theorem]{Lemma}
\newtheorem{corollary}[theorem]{Corollary}
\newtheorem{conjecture}[theorem]{Conjecture}
\newtheorem{question}[theorem]{Question}
\newtheorem{definition}[theorem]{Definition}
\newtheorem{example}[theorem]{Example}

\newtheorem{remark}[theorem]{Remark}
\def\Pf#1{{\noindent\bf Proof}.\setcounter{equation}{0}}
\def\>#1{{ $\Rightarrow$ }\setcounter{equation}{0}}
\def\<>#1{{ $\Leftrightarrow$ }\setcounter{equation}{0}}
\def\bskip#1{{ \vskip 20pt }\setcounter{equation}{0}}
\def\sskip#1{{ \vskip 5pt }\setcounter{equation}{0}}
\def\bg#1{\begin{#1}\setcounter{equation}{0}}
\def\ed#1{\end{#1}\setcounter{equation}{0}}
\def\KET{T^{^F\bot}\setcounter{equation}{0}}
\def\KEC{C^{\bot}\setcounter{equation}{0}}
%\def\KET{\mr{KerExt}_A^{i\ge 1}(T,-)\setcounter{equation}{0}}
%\def\KEC{\mr{KerExt}_A^{i\ge 1}(C,-)\setcounter{equation}{0}}
%\def\KEC{\mr{KerExt}_A^{i\ge 1}(-,C)\setcounter{equation}{0}}
%%%%%%%%%%

%----------------------------œÅ×¢ ÀàÐÍ ¿ªÊŒ-------------------------------------------------------------
\renewcommand{\thefootnote}{\fnsymbol{footnote}}
\setcounter{footnote}{0}
%²»Í¬µÄœÅ×¢·ûºÅ\footnote{·ûºÅÐÍœÅ×¢}£»²»Í¬µÄœÅ×¢·ûºÅ\footnote{·ûºÅÐÍœÅ×¢}
%\renewcommand{\thefootnote}{\arabic{footnote}}
%\setcounter{footnote}{0}
%²»Í¬µÄœÅ×¢·ûºÅ\footnote{Êý×ÖÐÍœÅ×¢}£»²»Í¬µÄœÅ×¢·ûºÅ\footnote{Êý×ÖÐÍœÅ×¢}
%\renewcommand{\thefootnote}{\roman{footnote}}
%\setcounter{footnote}{0}
%²»Í¬µÄœÅ×¢·ûºÅ\footnote{Ð¡ÐŽÂÞÂíÐÍœÅ×¢}£»²»Í¬µÄœÅ×¢·ûºÅ\footnote{Ð¡ÐŽÂÞÂíÐÍœÅ×¢}
%\renewcommand{\thefootnote}{\Roman{footnote}}
%\setcounter{footnote}{0}
%²»Í¬µÄœÅ×¢·ûºÅ\footnote{ŽóÐŽÂÞÂíÐÍœÅ×¢}£»²»Í¬µÄœÅ×¢·ûºÅ\footnote{ŽóÐŽÂÞÂíÐÍœÅ×¢}
%\renewcommand{\thefootnote}{\alph{footnote}}
%\setcounter{footnote}{0}
%²»Í¬µÄœÅ×¢·ûºÅ\footnote{Ð¡ÐŽ×ÖÄžÐÍœÅ×¢}£»²»Í¬µÄœÅ×¢·ûºÅ\footnote{Ð¡ÐŽ×ÖÄžÐÍœÅ×¢}
%\renewcommand{\thefootnote}{\Alph{footnote}}
%\setcounter{footnote}{0}
%²»Í¬µÄœÅ×¢·ûºÅ\footnote{ŽóÐŽ×ÖÄžÐÍœÅ×¢}£»²»Í¬µÄœÅ×¢·ûºÅ\footnote{ŽóÐŽ×ÖÄžÐÍœÅ×¢}
%²»Í¬µÄœÅ×¢·ûºÅ\footnote[8]{×Ô¶šÒåÐÍœÅ×¢}
%
%ÍšÑ¶×÷Õß£º\thanks{Corresponding author.} »á×Ô¶¯ÏÔÊŸ*
%
%Èç£º\author{author name\thanks{Corresponding author.}}
%----------------------------œÅ×¢ ÀàÐÍ œáÊø-------------------------------------------------------------------

%%%%%%%%%%%%%%%%%%%%%%%%%%%%%%%%%%%%%%%%%%%%%%%%%%%%%%%%%%%%%%%%%%%%%%%%%%%%%%%%%
%**************************±êÌâ¡¢ÕªÒª¡¢·ÖÀàºÅ¡¢¹ØŒü×Ö**************************

\title{\bf Support  $\tau$-tilting modules over  one-point extensions
\thanks{This work was partially supported by NSFC (Grant No. 11971225). } }
\footnotetext{
E-mail:~hpgao07@163.com, xiezongzhen3@163.com}
\smallskip
\author{\small Hanpeng Gao$^a$, Zongzhen Xie$^{b,} $\thanks{Corresponding author.}\\
{\it \footnotesize $^a$Department of Mathematics, Nanjing University, Nanjing 210093,  P.R. China}\\
{\it \footnotesize $^b$Department of Mathematics and Computer Science,  School of Biomedical Engineering and Informatics,}\\
{\it \footnotesize  Nanjing  Medical University, Nanjing 211166, P.R. China}}
\date{}
\maketitle
\baselineskip 15pt%16pt%14pt%15.5pt%\baselineskip  25.5pt %
%%%%%%%\hskip 18pt
%
% Abstract ------------------------------------------------------
%
\begin{abstract}
Let $B$ be the one-point extension algebra of  $A$ by an $A$-module $X$. We proved that every support $\tau$-tilting $A$-module can be extended to be a support $\tau$-tilting $B$-module by two different ways.  As a consequence,  it is shown that there is an  inequality $$|\stilt B|\geqslant 2|\stilt A|.$$
%\mskip\
\vspace{10pt}


%\sskip\


\end{abstract}
%\smallskip
%
\vskip 30pt
% ----------------------------------------------------------------------
%% ----------------------------------------------------------------------
%\def\baselinestretch{1}

\section{Introduction}
 Tilting modules are very important in the  representation  theory of finite dimensional algebras. Mutation is an effective way to construct a new tilting  module from  a given one.  Unfortunately,  mutation of tilting modules may not be realized.
 
In 2014,  Adachi, Iyama and Reiten \cite{AIR} introduced  the concept of support $\tau$-tilting module  as a generalization of tilting modules, 
and they showed that mutation of  support $\tau$-tilting modules is always possible.  The authors  also proved that  support $\tau$-tilting modules   are in bijection with some important classes in representation theory (such as, functorially finite torsion classes, 
 2-term silting complexes, and cluster-tilting objects in the cluster category).
 
 A new (support $\tau$)-tilting module can be constructed by algebra extensions. In \cite{AHT}, Assem, Happel and Trepode  studied how to extend and restrict tilting modules for one-point extension algebras
by a projective module. In  \cite{S}, Suarez generalized this result for the context of support
support $\tau$-tilting modules. More precisely, let $B=A[P]$ be the one-point extension of an algebra $A$
by a projective $A$-module $P$ and $e$ the identity of $A$. If $M$ is a  support
$\tau$-tilting $A$-module, then $\Hom_B(e B, M)\oplus S_a$ is a  support $\tau$-tilting $B$-module,
where $S_a$ is the simple module corresponding to the new point $a$ (see \cite[Theorem A]{S}). An example shown that $\Hom_B(e B, M)\oplus S_a$ may not be  a  support $\tau$-tilting $B$-module if $P$ is not projective (see \cite[Example 4.7]{S}). 

Bricks and semibricks are considered in \cite{Gab1962,Ringel1976}. An $A$-module $M$ is called $brick$ if $\Hom_\Lambda(M,M)$ is a $k$ division.  A $semibrick$ is a set consisting of isoclasses of pairwise Hom-orthogonal bricks. Let  $\mathcal{S}$ be a sembrick and $T(\mathcal{S})$  the smallest torsion class  containing $\mathcal{S}$.
In \cite{Asai2018}, the author called a semibrick $\mathcal{S}$ is $left$ $finite$ if   $T(\mathcal{S})$  is functorially finite and he also proved that 
 there exists a bijection $\Phi : \stilt A\mapsto \text{f}_L\text{-sbrick}A$ between the set of  support $\tau$-tilting $A$-modules  and   the set of left finite semibricks of $A$.

In this paper, we construct semibricks over the one-point extension $B$ of an algebra $A$ by an $A$-module $X$ (may not be projective) and use the bijection to get support $\tau$-tilting $B$-modules.

\begin{proposition}\label{1.1}{\rm (see Proposition \ref{3.2})}
Let $B$ be the one-point extension algebra of $A$ by an $A$-module $X$ and   $\mathcal{S}$ be a semibrick in $\mod A$. Then both $\mathcal{S}$ and $\mathcal{S}\cup S_a$ are semibricks in $\mod B$, where $S_a$ stands for the simple module corresponding to the extension point $a$.
\end{proposition}

Moreover,  it is shown that $\mathcal{S}$ is left finite implies $\mathcal{S}\cup S_a$ is also. We say an $A$-module $M$ is a support $\tau$-tilting module with respect to the semibrick $\mathcal{S}$ if  $\Phi(M)=\mathcal{S}$. 
As an application, we can construct  support $\tau$-tilting modules over one-point extensions from support $\tau$-tilting $A$-modules.
\begin{proposition}\label{1.2}{\rm (see Proposition \ref{1})}
Let $B$ be the  one-point extension algebra of $A$ by an $A$-module $X$
and $M\in\mod A$ be a support $\tau$-tilting module with respect to the semibrick $\mathcal{S}$. Then both $P(T(\mathcal{S}))$ and $P(T(\mathcal{S}\cup S_a))$ are support $\tau$-tilting $B$-modules.
\end{proposition}

As a consequence,  we have the following  inequality
\begin{corollary}
$|\stilt B|\geqslant 2|\stilt A|.$\end{corollary}

Moreover, we have 

\begin{theorem}{\rm (see Theorem \ref{3.9})} Let $B$ be  the one-point extension algebra of $A$ by an $A$-module $X$ and  $M$ be a support $\tau$-tilting module in $\mod A$. Then
\begin{enumerate}
\item [(1)] $M$ is a support $\tau$-tilting $B$-module. 
\item [(2)] Assume that  $M\in\mod A$ is  a support $\tau$-tilting module with respect to the semibrick $\mathcal{S}$, then $P(T(\mathcal{S}\cup S_a))$ has $M$ as direct summand.
\item[(3)]  If $X\in \Fac M$, then  $P_a\oplus M$ is a support $\tau$-tilting $B$-module. 
\item[(4)]  If $\Hom_A(X,\Fac M)=0$, then  $S_a\oplus M$ is a support $\tau$-tilting $B$-module. 
\end{enumerate}
\end{theorem}

Throughout this paper, all algebras will be  basic connected finite dimensional $k$-algebras over  an algebraically closed field $k$ and all modules are basic. Let $A$ be an algebra.  
 The category of finitely generated left $A$-modules will be  denote by $\mod A$ and   the Auslander-Reiten
  translation of $A$ will be  denote by $\tau$.  For  $M\in \mod A$, we denote by $\ind(M)$ the set of isoclasses of indecomposable direct summands of $M$, and by $\Fac M$  the full subcategory of $\mod A$ consisting
of modules isomorphic to factor modules of finite direct sums of copies of $M$.  For a finite set $J$,  $|J|$ stands for  the cardinality of $J$. In particular, we write $|M|=|\ind (M)|$. $\mathbb{N}$ will be the set of all   natural numbers.

\section{Preliminaries}\label{sect 2}

Let $A$ be an algebra.  In this section, we recall some definitions about support $\tau$-tilting modules and semibircks over $\mod A$.

\begin{definition}\label{2.1} {\rm (\cite[Definition 0.1]{AIR})}
Let $M\in\mod A$.
\begin{enumerate}
\item[(1)] $M$ is called {\it $\tau$-rigid} if $\Hom_A(M,\tau M)=0$.
\item[(2)] $M$ is called {\it $\tau$-tilting}  if it is $\tau$-rigid and $|M|=|A|$.
\item[(3)] $M$ is called {\it support $\tau$-tilting} if it is a $\tau$-tilting $A/\langle e\rangle$-module
where $e$ is an idempotent of $A$.
\end{enumerate}
\end{definition}

We will denote by  $\tilt A$ (respectively, $\stilt A$)  the set of isomorphism classes of $\tau$-tilting $A$-modules (respectively, support $\tau$-tilting $A$-modules).

\begin{definition} \label{a} {\rm (\cite[Definition 0.3]{AIR})
Let $(M,P)$ be a pair in $\mod A$ with $P$ projective.
\begin{enumerate}
\item[(1)] The pair $(M, P)$ is called a {\it $\tau$-rigid pair} if $M$ is $\tau$-rigid and $\Hom_A(P,M)=0$.
\item[(2)] The pair $(M, P)$ is called a {\it support $\tau$-tilting pair}
if it is $\tau$-rigid and $|M|+|P|=|A|$.
\end{enumerate}}
\end{definition}

Note that $(M,P)$ is a support $\tau$-tilting pair if and only if $M$ is a $\tau$-tilting $A/\langle e\rangle$-module,
where $eA\cong P$ \cite[Proposition  2.3]{AIR}. Hence, $M$ is a $\tau$-tilting $A$-module if and only if $(M,0)$ is a support $\tau$-tilting pair.

The following result is very useful.

\begin{lemma}\label{b}  {\rm (\cite[Proposition 5.8 ]{AS1981})} For $M\in\mod A$, $M$ is $\tau$-rigid if and only if $\Ext^1_A(M,\Fac M)=0$.
\end{lemma}

\begin{definition}\label{2.2}{\rm (\cite[Definition 2.1]{Asai2018})} Let $\mathcal{S} \subseteq\mod A$. $\mathcal{S}$ is called a $semibrick$ if
 $$\Hom_A(S_i,S_j)=\begin{cases}
\text{$k$-division algebra}& if ~~i=j\\
0&~~ otherwise
\end{cases} $$
for any $S_i, S_j\in \mathcal{S}$.  
\end{definition}

 By Schur's Lemma,  a set of isoclasses of some simple modules is a semibrick.

Let $\mathcal{Y}$ be a full subcategory of $\mod A$ and $M\in \mod A$.   A  homomorphism $f_M:M\to Y_M$ is called left $\mathcal{Y}$-approximation of $M$
with $Y_M\in\mathcal{Y}$ if  any morphism $f:M\to Y$ with $Y\in\mathcal{Y}$  factors through $f_M$.  We say that  $\mathcal{Y}$ is $covariantly~ finite$ if for any $M\in \mod\Lambda$, there exists a left $\mathcal{Y}$-approximation of $M$. 
Dually, we can define the concepts of  right $\mathcal{Y}$-approximation of $M$ and $contravariantly~finite$ subcategories. $\mathcal{Y}$ is called $functorially~ finite$ if it is both  covariantly finite and contravariantly finite.  

A $torsion$ $class$ of $\mod A$ is a  full subcategory  of $ \mod A$  closed under  images, direct sums, and extensions.  Recall that a semibrick $\mathcal{S}$  of $\mod A$ is $left~ finite$\cite{Asai2018}  if   $T(\mathcal{S})$    is functorially finite, where  $T(\mathcal{S})$  is the smallest torsion class  containing $\mathcal{S}$. The set of all left finite semibricks of $\mod A$ will be denoted by $\f_L$-$\sbrick A$.

The following result states the relationship between $\stilt A$ and $\f_L$-$\sbrick A$.

 \begin{theorem}\label{2.3}{\rm \cite[Theorem 1.3(2)]{Asai2018}}  there exists a bijection

 $$\Phi : \stilt A\mapsto \f_L\text{-}\sbrick A$$
  given by $M\mapsto \ind(M/\rad_\Gamma M)$ where $\Gamma=\End_A(M)$.\end{theorem}

 Recall that  $M\in\mod A$ is called $sincere$ if every simple $A$-module  appears as a composition factor in $M$.  A $\tau$-tilting $A$-module is exactly a sincere support $\tau$-tilting.  We say a semibrick $\mathcal{S}$  of $\mod A$ is $sincere$  if   $T(\mathcal{S})$  is sincere.
 Let $\sf_L$-$\sbrick A$ stand for all sincere left finite semibricks of $\mod A$. We have the following result due to Asai in \cite{Asai2018}.
 
 \begin{corollary}\label{2.4} There exists a bijection $\Phi : \tilt A\mapsto  \sf_L\text{-}\sbrick A.$
 \end{corollary} 
 
\section{Main results}\label{sect 3}

Let  $X\in \mod A$. The {\it one-point extension} of $A$ by $X$ is defined as the following matrix algebra
\begin{center}
$B=\left(\begin{matrix}
A &X\\
0 & k\\
\end{matrix}\right)$
\end{center}
with the ordinary matrix addition and the multiplication  induced by the module structure of $X$. We write $B:=A[X]$ with $a$ the extension point.  All $B$-modules can be viewed as ${M\choose k^n}_f$ where $M\in\mod A$, $n\in \mathbb{N}$  and $f\in \Hom_A(X\otimes_k k^n, M)$(see, \cite[\uppercase\expandafter{\romannumeral15.1}]{SS2007}). In particular,$S_a={0\choose k}_0$ and $P_a={X\choose k}_{1}$. Moreover,
the morphisms from ${M\choose  k^n}_{f}$ to ${M' \choose k^{n'}}_{f'}$ are pairs of ${\alpha\choose\beta}$
such that the following diagram
$$\xymatrix{X\otimes_k k^n\ar[d]_{X\otimes \beta}\ar[rr]^f&&M\ar[d]^\alpha\\
X\otimes_k k^{n'}\ar[rr]^{f'}&&M'\\}$$
commutes, where $\alpha\in\Hom_\Lambda(M,M')$ and $\beta\in\Hom_\Gamma(k^n,k^{n'})$.
A sequence
$$0 \to {M_1\choose k^{n_1}}_{f_1}\stackrel{{{\alpha_1\choose \beta_1}}}{\longrightarrow}
{M_2\choose k^{n_2}}_{f_2}\stackrel{{{\alpha_2\choose\beta_2}}}{\longrightarrow}{M_3\choose k^{n_3}}_{f_3}\to 0$$
in $\mod B$ is exact if and only if
$$0 \to M_1\stackrel{\alpha_1}{\longrightarrow} M_2\stackrel{\alpha_2}{\longrightarrow}M_3\to 0$$
is exact in $\mod A$ and
$$0 \to k^{n_1}\stackrel{\beta_1}{\longrightarrow}  k^{n_2}\stackrel{\beta_2}{\longrightarrow} k^{n_3}\to 0$$
is exact in $\mod k$.

\begin{lemma}\label{3.1} For any $M\in\mod A$, we have 
\begin{enumerate}
\item[(1)] $\Hom_B(S_a,M)=0$.
\item[(2)] $\Hom_B(M,S_a)=0$.
\end{enumerate}\end{lemma}
\begin{proof} It is clear since $_BM\cong {M\choose 0}_0$. Hence, $\Hom_B(S_a,M)\cong\Hom_B({0\choose k}_0,{M\choose 0}_0)=0$. Similarly, we can get $\Hom_B(M,S_a)=0$.
\end{proof}

\begin{proposition}\label{3.2}
Let  $\mathcal{S}$ be a semibrick in $\mod A$. Then both $\mathcal{S}$ and $\mathcal{S}\cup S_a$ are semibricks in $\mod B$.
\end{proposition}
\begin{proof}
It follows from Lemma \ref{3.1}.
\end{proof}

\begin{lemma}\label{3.3}
Let  $\mathcal{S}$ be a semibrick in $\mod A$. Then  $$T(\mathcal{S}\cup S_a)=\{{M\choose k^n}_f \mid \forall   n\in \mathbb{N}, ~ M\in T(\mathcal{S})~\text{ and }~f\in \Hom_A(X\otimes_k k^n, M)\}.$$
\end{lemma}
\begin{proof} Since $\mathcal{S}$ and $S_a$ belong to $T(\mathcal{S}\cup S_a)$, we have $\{{M\choose 0}_0\mid M\in T(\mathcal{S})\}\subset T(\mathcal{S}\cup S_a)$ and  ${0\choose k^n}\in T(\mathcal{S}\cup S_a)$  for all $n\in \mathbb{N}$.
Note that $ \forall   n\in \mathbb{N}, ~ M\in T(\mathcal{S})~\text{ and }~f\in \Hom_A(X\otimes_k k^n, M)$, 
 there exists the following exact sequence in $\mod B$
$$0\to {M\choose 0}_0\to {M\choose k^n}_f\to{0\choose k^n}_0\to 0.$$ 
this implies  ${M\choose k^n}_f\in T(\mathcal{S}\cup S_a)$.
It is clear that $\{{M\choose k^n}_f \mid \forall   n\in \mathbb{N}, ~ M\in T(\mathcal{S})~\text{ and }~f\in \Hom_A(X\otimes_k k^n, M)\}$  is closed under image, direct sum and extension. Thus it
is a torsion class. Hence $T(\mathcal{S}\cup S_a)=\{{M\choose k^n}_f \mid \forall   n\in \mathbb{N}, ~ M\in T(\mathcal{S})~\text{ and }~f\in \Hom_A(X\otimes_k k^n, M)\}.$

\end{proof}

\begin{proposition}\label{3.4}
Let  $\mathcal{S}$ be a semibrick in $\mod A$. If  $\mathcal{S}$ is left finite, then   $\mathcal{S}\cup S_a$ is also.
\end{proposition}

\begin{proof}
We only show that $T(\mathcal{S}\cup S_a)$ is covariantly~ finite. It is dually to prove $T(\mathcal{S}\cup S_a)$ is contravariantly finite.

Let ${M\choose k^n}_f\in\mod B$. Then  $M$ has a left $T(\mathcal{S})$-approximation $h_M:M\to Z_M$ in $\mod A$ since $T(\mathcal{S})$ is covariantly finite.    Take $g=~h_M\circ f$. The following commutative diagram
$$\xymatrix@R=10pt{X\otimes_k k^n\ar^{~~f}[r]\ar@{=}[d]&M\ar^{h_M}[d]\\
X\otimes_k k^n\ar^{~~~g}[r]&Z_M
}$$
implies that   ${f_M\choose 1}$ is a morphism from ${M\choose k^n}_f$ to ${Z_M\choose k^n}_g$.
 Next, we will show that ${f_M\choose 1}$ is  left $T(\mathcal{S}\cup S_a)$-approximation of ${M\choose k^{n}}_{f}$. For any ${M_1\choose k^{n_1}}_{f_1}\in T(\mathcal{S}\cup S_a)$ and morphism ${a\choose b}:{M\choose k^n}_f\to {M_1\choose k^{n_1}}_{f_1}$, there  is  a morphism $h':Z_M\to  M_1$ such that $a=h'\circ h_M$ since $h_M$ is a  left approximation.  Note that there exists a commutative diagram
 $$\xymatrix@R=10pt{X\otimes_k k^n\ar^{~~f}[r]\ar_{X\otimes b}[d]&M\ar^{a}[d]\\
X\otimes_k k^{n_1}\ar^{~~~f_1}[r]&M_1
}$$
 
that is $a\circ f=f_1\circ (X\otimes b)$. Therefore, $$f_1\circ (X\otimes b)=a\circ f=h'\circ h_M \circ f=h'\circ g,$$
that is, the following diagram
 $$\xymatrix@R=10pt{X\otimes_k k^n\ar^{~~g}[r]\ar_{X\otimes b}[d]&Z_M\ar^{h'}[d]\\
X\otimes_k k^{n_1}\ar^{~~~f_1}[r]&M_1
}$$
commutates.
Hence,    ${h'\choose b}$ is a morphism from ${ Z_M\choose k^n}_g$ to ${M_1\choose k^{n_1}}_{f_1}$, and the following equation holds
$${h'\choose b}\circ{h_M\choose 1}={h'\circ h_M\choose b}={g\choose b}.$$
By Lemma \ref{3.3},  ${Z_M\choose k^n}_g\in T(\mathcal{S}\cup S_a)$ since $Z_M\in T(\mathcal{S})$. Thus, we were done.
\end{proof}

The following result can be found  immediately.
\begin{corollary}
Let  $\mathcal{S}$ be a semibrick in $\mod A$. If  $\mathcal{S}$ is sincere left finite, then   $\mathcal{S}\cup S_a$ is also.
\end{corollary}

Let $\mathcal{F}$ be a full subcategory of $\mod A$. An $A$-module $M$ is called ${\bf Ext\text{-}projective}$ in $\mathcal{F}$  if $\Ext^1_A(M,F)=0$ for all $F\in \mathcal{F}$.  If $\mathcal{F}$ is  functorially finite in $\mod A$ , then there are only  finitely many indecomposable $\Ext$-projective modules in $\mathcal{F}$  up to isomorphism. In this case, we will denote by $P(\mathcal{F})$ the direct sum of all Ext-projective modules in $\mathcal{F}$ up to isomorphism.

\begin{definition}  We say that an $A$-module $M$ is a support $\tau$-tilting module with respect to the semibrick $\mathcal{S}$ if  $\Phi(M)$=$\mathcal{S}$.\end{definition}

Now, we can construct support $\tau$-tilting $B$-modules from support $\tau$-tilting $A$-modules.

\begin{proposition}\label{1}
Let $M\in\mod A$ be a support $\tau$-tilting module with respect to the semibrick $\mathcal{S}$. Then both $P(T(\mathcal{S}))$ and $P(T(\mathcal{S}\cup S_a))$ are support $\tau$-tilting $B$-modules. Moreover, if $M$ is $\tau$-tilting, then $P(T(\mathcal{S}\cup S_a))$  is also.
\end{proposition}

\begin{proof} Since $M$ is a  support $\tau$-tilting module,  we have $\mathcal{S}$ is a left finite semibrick of $\mod A$ by Theorem \ref{2.3}. Hence,  $\mathcal{S}$ is also a left finite semibrick of $\mod B$. Moreover,  we have $\mathcal{S}\cup S_a$  is  a left finite semibrick of $\mod B$ by Proposition \ref{3.4}. Therefore, $T(\mathcal{S})$ and $T(\mathcal{S}\cup S_a)$ are   functorially~finite torsion classes. Hence,   $P(T(\mathcal{S}))$  and $P(T(\mathcal{S}\cup S_a))$ are support $\tau$-tilting $B$-module by \cite[Theorem 2.7]{AIR}).
\end{proof}
As a consequence,  we have the following  inequality
\begin{corollary}
$|\stilt B|\geqslant 2|\stilt A|.$\end{corollary}

Applying Proposition \ref{1}, we can give those forms of support $\tau$-tilting $B$-module under certain conditions.

\begin{theorem}\label{3.9} Let $B$ be  the one-point extension of $A$ by $X$ and  $M$ be a support $\tau$-tilting module  in $\mod A$. Then
\begin{enumerate}
\item [(1)] $M$ is a support $\tau$-tilting $B$-module. 
\item [(2)] Assume that  $M\in\mod A$ is  a support $\tau$-tilting module with respect to the semibrick $\mathcal{S}$, then $P(T(\mathcal{S}\cup S_a))$ has $M$ as direct summand.
\item[(3)]  If $X\in \Fac M$, then  $P_a\oplus M$ is a support $\tau$-tilting $B$-module.
\item[(4)]  If $\Hom_A(X, \Fac M)=0$, then  $S_a\oplus M$ is a support $\tau$-tilting $B$-module.
\end{enumerate}
\end{theorem}
\begin{proof}
Assume that  $M\in\mod A$ is  a support $\tau$-tilting module with respect to the semibrick $\mathcal{S}$, then $T(\mathcal{S})=\Fac M$\cite[Lemma 2.5(5)]{Asai2018}. 

Note that $ \forall   n\in \mathbb{N}, ~ M'\in T(\mathcal{S})~\text{ and }~f\in \Hom_A(X\otimes_k k^n, M')$, 
 there exists the following exact sequence in $\mod B$
$$0\to {M'\choose 0}_0\to {M'\choose k^n}_f\to{0\choose k^n}_0\to 0\eqno{(1)}.$$ 

For any $Y\in \mod B$, applying the functor $\Hom_B(Y,-)$ to (1.1), we have the following exact sequence

$$ \Ext^1_B(Y, {M'\choose 0}_0)\to \Ext^1_B(Y, {M'\choose k^n}_f)\to\Ext^1_B(Y,{0\choose k^n}_0)= 0\eqno{(1.2)}.$$

(1) By  Proposition \ref{1},  $P(T(\mathcal{S}))=P(\Fac M)=M$ is a support $\tau$-tilting $B$-module.

(2) Putting $Y=$~$ _BM\cong {M\choose 0}_0$ in (1.2), we have $\Ext^1_B(M, {M'\choose 0}_0)\cong\Ext^1_A(M, M')=0$ by Lemma \ref{b} , and hence $\Ext^1_B(M, {M'\choose k^n}_f)=0$. By Lemma \ref{3.3},  $M$ is a $\Ext$-projective object in $T(\mathcal{S}\cup S_a)$. Therefore,   $P(T(\mathcal{S}\cup S_a))$ has $M$ as direct summand.

(3) If $X\in \Fac M$, then $P_a\in T(\mathcal{S}\cup S_a)$ by Lemma \ref{3.3}. Hence,  $P_a\oplus M$ is a direct summand of $P(T(\mathcal{S}\cup S_a))$ by (2).
In particular,  $P_a\oplus M$  is a $\tau$-rigid $B$-module.
 Suppose that $(M,P)$ is a support $\tau$-tilting pair in $\mod A$. Hence, $\Hom_A(P,\Fac M)=0$ because $\Hom_A(P, M)=0$. This implies
$\Hom_B(P, P_a)\cong \Hom_A(P, X)=0$. Therefore, $(P_a\oplus M,P)$ is a support $\tau$-tilting pair in $\mod B$ since $|P_a\oplus M|+|P|=1+|A|=|B|$.

(4)Note that there is an exact  sequence in $\mod B$,
$$0\to{X\choose 0}\cong X \stackrel{f}\to P_a\to S_a\to 0.$$
For any $Y'\in \Fac M$, applying $\Hom_B(-,Y')$ to it, we have the following exact sequence,
$$\Hom_B(X, Y')\to\Ext^1_B(S_a,Y') \to \Ext^1_B(P_a,Y')=0.$$
Since $\Hom_A(X, \Fac M)=0$, we have  $\Hom_B(X, Y')=0$. Hence,  $\Ext^1_B(S_a,Y')=0$. Thus,
$\Ext^1_B(S_a,\Fac M)=0$.  Putting $Y=S_a$ in (1.2), we have  $\Ext^1_B(S_a, {M'\choose k^n}_f)=0$. By Lemma \ref{3.3},  $S_a$ is a $\Ext$-projective object in $T(\mathcal{S}\cup S_a)$. Therefore,   $P(T(\mathcal{S}\cup S_a))$ has $S_a\oplus M$ as direct summand. This implies   $S_a\oplus M$ is a $\tau$-rigid $B$-module. Suppose that $(M,P)$ is a support $\tau$-tilting pair in $\mod A$. It is easy to get  $(S_a\oplus M,P)$ is a support $\tau$-tilting pair in $\mod B$ since $\Hom_B(P, S_a)=0$ and  $|S_a\oplus M|+|P|=1+|A|=|B|$. Hence, $S_a\oplus M$ is a support $\tau$-tilting $B$-module. 

\end{proof}

\begin{corollary}\label{3.10} Let $B$ be  the one-point extension of $A$ by $X$ and  $M\in\mod A$ be a  $\tau$-tilting module. Then
\begin{enumerate}
\item[(2)]  If $X\in \Fac M$, then  $P_a\oplus M$ is a $\tau$-tilting $B$-module. 
\item[(3)]  If $\Hom_A(X, \Fac M)=0$, then  $S_a\oplus M$ is a  $\tau$-tilting $B$-module. 
\end{enumerate}
\end{corollary}

\begin{example}\label{4.3}
{\rm Let $A$ be a finite dimensional $k$-algebra given by the quiver
$$2 {\longrightarrow} 3.$$ Considering the one-point extension of $A$ by the simple module corresponding to the point $2$,
the algebra $B=A[2]$ is given by the quiver
$$1 \stackrel{\alpha}{\longrightarrow} 2 \stackrel{\beta}{\longrightarrow} 3$$ with the relation  $\alpha\beta=0$.
 The  Hasse quiver of $A$ is as follows (semibricks be remarked by red).
\[\xymatrix{{\smallmatrix Q(\stilt A):
\endsmallmatrix} &T_1={\smallmatrix \color{red}{2}\\3
\endsmallmatrix}{\smallmatrix \color{red}{3}
\endsmallmatrix}\ar[r]\ar[d] &T_2={\smallmatrix  \color{red}{3}
\endsmallmatrix}\ar[r]&T_3={\smallmatrix  \color{red}{0}
\endsmallmatrix}\\
 &T_4={\smallmatrix  \color{red}{2}\\ \color{red}{3}
\endsmallmatrix}{\smallmatrix 2
\endsmallmatrix}\ar[r]&T_5={\smallmatrix  \color{red}{2}
\endsmallmatrix}.\ar[ru]&
}\]

(1) All support $\tau$-tilting $A$-modules $T_i(i=1,2,3,4,5)$ are   support $\tau$-tilting $B$-modules by Theorem   \ref{3.9}(1).

(2) Since $2\in\Fac T_i(i=1,4,5)$, we have three support $\tau$-tilting $B$-modules 
$P_1\oplus T_1$, $P_1\oplus T_4$,$P_1\oplus T_5$ by Theorem   \ref{3.9}(3) . Moreover, $P_1\oplus T_1$, $P_1\oplus T_4$  are $\tau$-tilting $B$-modules since $T_1$, $T_4$ are $\tau$-tilting $A$-modules by Corollary \ref{3.10}. 

(3) Since $\Hom_A(2,\Fac T_i)=0 (i=2,3)$, we have  two  support $\tau$-tilting $B$-modules 
$S_1\oplus T_2$, $S_1\oplus T_3$ by Theorem   \ref{3.9}(4) .

In fact, the  Hasse quiver $Q(\stilt B)$ is as follows.
\[\xymatrix@C=30pt{ {\smallmatrix Q(\stilt B):
\endsmallmatrix} &{\smallmatrix \color{red}{1}\\ 2
\endsmallmatrix}{\smallmatrix \color{red}{2}\\ \color{red}{3}
\endsmallmatrix}{\smallmatrix 2
\endsmallmatrix}\ar[r]\ar[rrd]&{\smallmatrix \color{red}{1}\\2
\endsmallmatrix}{\smallmatrix \color{red}{2}
\endsmallmatrix}\ar[r]\ar [rrd]&{\smallmatrix \color{red}{1}\\\color{red}{2}
\endsmallmatrix}{\smallmatrix 1
\endsmallmatrix}\ar[r]&{\smallmatrix \color{red}{1}
\endsmallmatrix}\ar[rd]&\\
{\smallmatrix \color{red}{1}\\2
\endsmallmatrix}{\smallmatrix \color{red}{2}\\3
\endsmallmatrix}{\smallmatrix \color{red}{3}
\endsmallmatrix}\ar[r]\ar[ur]\ar[rd]&{\smallmatrix \color{red}{1}\\ \color{red}{2}
\endsmallmatrix}{\smallmatrix 1
\endsmallmatrix}{\smallmatrix \color{red}{3}
\endsmallmatrix}\ar[r]\ar[rru]&{\smallmatrix \color{red}{1}
\endsmallmatrix}{\smallmatrix \color{red}{3}
\endsmallmatrix}\ar[rru]\ar[rrd]&{\smallmatrix \color{red}{2}\\\color{red}{3}
\endsmallmatrix}{\smallmatrix 2
\endsmallmatrix}\ar[r]&{\smallmatrix \color{red}{2}
\endsmallmatrix}\ar[r]&{\smallmatrix \color{red}{0}
\endsmallmatrix}\\
&{\smallmatrix \color{red}{2}\\3
\endsmallmatrix}{\smallmatrix \color{red}{3}
\endsmallmatrix}\ar[rrr]\ar[rru]&&&{\smallmatrix \color{red}{3}
\endsmallmatrix}\ar[ru]&
}\]
}.
\end{example}

\end{document}